\documentclass[11pt,a4paper]{article}
\usepackage{indentfirst}
\usepackage{amsfonts}
\usepackage{amssymb}
\usepackage{mathrsfs}
\usepackage{amsmath}
\usepackage{amsthm}
\usepackage{enumerate}
\usepackage{cite}
\usepackage{mathrsfs}
\allowdisplaybreaks
\usepackage{geometry}
\usepackage{ifpdf}
\ifpdf
  \usepackage[colorlinks=true,linkcolor=blue,citecolor=red, final,backref=page,hyperindex]{hyperref}
\else
  \usepackage[colorlinks,final,backref=page,hyperindex,hypertex]{hyperref}
\fi

\usepackage{times}
\usepackage{indentfirst}
\usepackage{latexsym}
\usepackage[all]{xy}
\def\<{\langle}
\def\>{\rangle}

\def\c{\cdot}

\newcommand{\bmx}{\begin{pmatrix}}
\newcommand{\emx}{\end{pmatrix}}

\newtheorem{thm}{Theorem}[section]
\newtheorem{lem}[thm]{Lemma}
\newtheorem{cor}[thm]{Corollary}
\newtheorem{pro}[thm]{Proposition}
\newtheorem{ex}[thm]{Example}

\theoremstyle{definition}
\newtheorem{defi}{Definition}[section]

\theoremstyle{remark}
\newtheorem{rmk}{Remark}[section]
\begin{document}
\title{\sf  Mock-Lie bialgebras and mock-Lie analogue of the classical Yang-Baxter equation}

\author{ K. Benali$^{1}$
 \footnote { E-mail: karimabenali172@yahoo.fr},\ { T. Chtioui$^{1}$
 \footnote { E-mail: chtioui.taoufik@yahoo.fr}, A. Hajjaji$^{1}$
 \footnote { E-mail: atefhajjaji100@gmail.com},
\ and  S. Mabrouk$^{2}$
 \footnote { E-mail: mabrouksami00@yahoo.fr (Corresponding author)}
}\\
{\small 1.  University of Sfax, Faculty of Sciences,  BP
1171, 3038 Sfax, Tunisia.} \\
{\small 2.  University of Gafsa, Faculty of Sciences, 2112 Gafsa, Tunisia.}}
\date{}
\maketitle
\begin{abstract}
The aim of this paper is to introduce the notion of a mock-Lie bialgebra which is equivalent to a Manin triple of mock-Lie
algebras. The study of a special case  called coboundary mock-Lie bialgebra leads to the introduction the mock-Lie Yang-Baxter equation on a mock-Lie algebra which is an analogue of the classical
Yang-Baxter equation on a Lie algebra. Note that a skew-symmetric solution of mock-Lie
Yang-Baxter equation gives a mock-Lie bialgebra.  Finally, the notation of $\mathcal O$-operators are studied to construct skew-symmetric solution  of mock-Lie Yang-Baxter equation. 
\end{abstract}
\textbf{M. S. C} (2020):16W10, 16T10, 16T15, 16T25, 17B38.

\numberwithin{equation}{section}

\tableofcontents

%%%%%%%%%%%%%%%%%%%%%%%%%%%%%%%%%%%%%%%%%%%
\section{Introduction}
%%%%%%%%%%%%%%%%%%%%%%%%%%%%%%%%%%%%%%%%%%%
A while ago, a new class of algebras emerged in the literature the so called mock-Lie algebras. These are commutative algebras satisfying the Jacobi identity. They were appeared for the first time in \cite{B.F}
and since then a lot of works are done on this subject, note for example \cite{P.Z,C.G}. These algebras live a dual life: as member of a very particular class of Jordan algebras and as strange cousins of Lie algebras.\\
%For example take a finite dimensional vector space $V$ with a given algebraic structure, this can be acheived by equipping the dual space $V^{*}$ with the same algebraic structure and a set of compatibility conditions between the structures on $V$ and those on $V^{*}$.
%Among the well-known bialgebra structures, we have the associative bialgebra and infinitesimal bialgebra introduced in \cite{M.G,S.G}.
% Note that these two structures have the same associative multiplications on $V$ and $V^{*}$.
%They are distinguished only by the compatibility conditions, with the comultiplication acting as
%an homomorphism (respectively a derivation) on the multiplication for the associative bialgebra (respectively
%the infinitesimal bialgebra). In general, it is quite common to have multiple bialgebra structures that differ only by their compatibility conditions.
%A good compatibility condition is prescribed on one hand by a strong motivation and potential
%applications, and on the other hand by a rich structure theory and effective constructions.
%
%In the context of Lie algebra, the most common bialgebra structure is the
%Lie bialgebra, 
The theory of Lie bialgebra and Poisson Lie groups dates back to the early 80s.

Poisson Lie groups are Lie groups equipped with an additional structure, a Poisson bracket satisfying a compatibility condition with the group multiplication. The infinitesimal object associated with a poisson Lie group is the tangent vector space at the origin of the group, which is in a naturel way a Lie algebra $\mathfrak{g}$,  see for instance\cite{R.J.B,G.C.R}.The Poisson structure on the group induces on the Lie algebra an additional structure which is nothing but a Lie algebra structure on the dual vector space $\mathfrak{g}^*$ satisfying a compatibility condition with the Lie bracket on $\mathfrak{g}$ itself. Such a Lie algebra together with its additional structure is called a Lie bialgebra. So
a bialgebra structure on a given algebra is obtained by a corresponding set of comultiplication together with the set of compatibility conditions between multiplication and comultiplication \cite{B.L}.
For example take a finite dimensional vector space $V$ with a given algebraic structure, this can be acheived by equipping the dual space $V^{*}$ with the same algebraic structure and a set of compatibility conditions between the structures on $V$ and those on $V^{*}$.
Among the well-known bialgebra structures, we have the associative bialgebra and infinitesimal bialgebra introduced in \cite{M.G,S.G}.
 Note that these two structures have the same associative multiplications on $V$ and $V^{*}$.
They are distinguished only by the compatibility conditions, with the comultiplication acting as
an homomorphism (respectively a derivation) on the multiplication for the associative bialgebra (respectively
the infinitesimal bialgebra). In general, it is quite common to have multiple bialgebra structures that differ only by their compatibility conditions.
A good compatibility condition is prescribed on one hand by a strong motivation and potential
applications, and on the other hand by a rich structure theory and effective constructions.
See also \cite{D.B,Bai2,BasdouriFadousMabroukMakh,Dassoundo1,Drinfeld,FadousSamiMakhlouf,Makhlouf1,MaMakhloufSilvestrov, Zhelyabin,Zhelyabin2,Zhelyabin3,Dassoundo2,V.A} for more details.

One reason for the usefulness of the Lie bialgebra is that it has a coboundary theory, which
leads to the construction of Lie bialgebras from solutions of the classical Yang-Baxter equations.
The origin of the Yang-Baxter-equations is purely physics. They were first introduced by Baxter, McGuire, and Yang in \cite{R.J1,R.J2, C.N}.
Later on,  this equation attracts the attention of scientists and becomes one of the most basic equation in mathematical physics \cite{B.G.T,B.B.G}.
Namely it plays a crucial role for introducing the theory of quantum groups. This exceptional importance can be seen  in many other domaines like: quantum groups, knot theory, braided categories, analysis of integrable systems, quantum mechanics, non-commutative descent theory, quantum computing, non-commutative geometry, etc. The various forms of the Yang-Baxter-equation and some of their uses in physics are summarized in \cite{J.H}. 
Many scientists have found solutions for the Yang-Baxter equation, however the full classification of its solutions remains an open problem. In the theory of Lie bialgebras, it is essential to consider the coboundary case, which is
related to the theory of the classical Yang-Baxter equation\cite{B.L, B, Bai1, Hou}. We aim to have an anlogue in th mock-Lie case. 

This paper is organized as follows:
In Section $2$ we recall some basic definitions and constructions about mock-Lie algebras. Section $3$ deals with matched pairs, manin triple and mock-Lie bialgebras. In Section $4$, we introduce and develop the notion of coboundary mock-Lie bialgebras and the mock-Lie Yang-Baxter equation. In Section $5$, we give the $\mathcal{O}$-operators of mock-Lie algebras and construct a solution mock-Lie Yang-Baxter equation.

 Unless otherwise specified, all the vector spaces and algebras are finite dimensional over a field $\mathbb{K}$ of characteristic zero. 
\vspace{0.2 cm}

\textbf{Notations.}
Let $V$ and $W$ be two vector spaces:
\begin{enumerate}
    \item Denote by $\tau:V\otimes W \rightarrow W\otimes V$ the switch  isomorphism, $\tau(v\otimes w)=w\otimes v$.
    \item For a linear map $\Delta: V \rightarrow \otimes^{2}V$ , we use Sweedler's notation $\Delta(x) = \sum_{(x)}
 x_1\otimes 
 x_2$ for $x \in V$. We will often omit the summation sign $\sum_{(x)}$
to simplify the notations.
    \item Denote by $V^*=Hom(V,\mathbb{K})$ the linear dual of $V$. For $\varphi\in V^*$ and $u\in V$, we write 
     $\langle \varphi,u\rangle:=\varphi(u)\in \mathbb{K}$. 
    \item For a linear map $\phi : V \rightarrow W$, we define the map $\phi^{*} :W^{*} \rightarrow V^{*}$ by
\begin{equation}
    \langle \phi^{*}(\xi), v \rangle=\langle \xi, \phi(v) \rangle,\;\forall v\in V,\xi\in W^{*}.
\end{equation}
    \item For an element $x$ in a mock-Lie algebra $(A,\bullet)$ and $n\geq 2$, define the adjoint map $L(x):\otimes^{n}A\rightarrow \otimes^{n}A$ by 
    \begin{equation}\label{adjointaction}
      L(x)(y_1\otimes \cdots \otimes y_n)=\sum_{i=1}^{n} y_1\otimes \cdots \otimes y_{i-1}\otimes x\bullet y_i\otimes y_{i+1}\otimes  \cdots \otimes y_n 
    \end{equation}
    for all $y_1,\ldots,y_n\in A.$
    Conversely, given $Y=y_1\otimes \cdots \otimes y_n$, we define $L(Y):\mathfrak{g}\rightarrow \otimes^{n}\mathfrak{g}$ by
    \begin{equation*}
        L(Y)(x)=L(x)(Y), \ \text{for}\  x\in \mathfrak{g}.
    \end{equation*}
\end{enumerate}
%%%%%%%%%%%%%%%%%%%%%%%%%%%%%%%%%%%%%%%%%%%%
\section{Preliminaries}
%%%%%%%%%%%%%%%%%%%%%%%%%%%%%%%%%%%%%%%%%
In this section, we provide some preliminaries  about mock-Lie algebras and left  mock-pre-Lie algebras. Our main references are \cite{Agore1,B.F,Attan,Baklouti}.

\begin{defi}
A \textbf{mock-Lie} algebra  is a pair  $(A,\bullet)$ consisting of a vector space $A$ together with a multiplication $\bullet: A \otimes A \to A$ satisfying 
\begin{align}
  &  x \bullet y= y \bullet x, \quad \mathrm{(Commutativity)},\\ 
 &   x\bullet(y\bullet z)+y\bullet(z\bullet x)+z\bullet(x\bullet y)=0,  \quad \mathrm{ (Jacobi\ identity)},  \label{Jacoby id}
\end{align}
   for any  $x,y,z\in A$. 
The Jacobi identity \eqref{Jacoby id} is equivalent to 
\begin{align}
&  x\bullet (y \bullet z)=-(x\bullet y)\bullet z -y \bullet (x \bullet z).  
\end{align}
 In other words, the left multiplication $L: A \to End(A)$ defined by $L(x)y=x \bullet y$, is anti-derivation on $A$.  Recall that   a linear map  $D: A \to  A$ is called  \textbf{anti-derivation}  if  for all $x,y\in A$, 
    \begin{equation*}
        D(x \bullet y)=-D(x)\bullet y-x\bullet D(y). 
    \end{equation*}

\end{defi}

\begin{ex}
Let $A$ be a $4$-dimensional vector space with basis $\mathcal{B}=\{e_1, e_2, e_3, e_4\}$. Then 
$(A,\bullet)$ is a mock-Lie algebra where  the product $\bullet$ is defined on the basis $\mathcal{B}$ by 
$e_1 \bullet e_1 = e_2$,\  $e_1 \bullet e_3 = e_4$.
\end{ex}
 
\begin{ex}
Recall that 
an  \textbf{anti-associative} algebra is a pair $(A, \star)$ consisting of a vector space $A$ together with a product $\star: A \otimes A \to A$ such that the anti-associator vanishes, i.e, 
\begin{align}
  Aass(x,y,z):=(x\star y)\star z+x\star(y\star z)=0,\quad \forall x,y,z \in A.   
\end{align}
Let $(A, \star)$ be an anti-associative algebra. Then, 
$(A,\bullet)$ is a mock-Lie algebra, where 
$x \bullet y := x\star y + y\star x,\; \forall x, y \in  A$.
\end{ex}
 Now, we recall the definition of representations of a mock-Lie algebra.
\begin{defi}
A \textbf{representation} of a mock-Lie  algebra $(A,\bullet)$ is a pair $(V,\rho)$ where $V$ is a vector space and $\rho:A \rightarrow End(V)$ is a linear map such that the following equality holds, for all $x,y\in A$,
\begin{equation}
    \rho(x\bullet y)=-\rho(x)\rho(y)-\rho(y)\rho(x).
\end{equation}
\end{defi}
\begin{ex}
Let $(A,\bullet)$ be a mock-Lie algebra. Then $(A,L)$ is a representation of $A$ on itself called the adjoint representation.
\end{ex}
An equivalent  characterisation of representations on mock-Lie algebras is given in the following.
\begin{pro}
Let $(A, \bullet)$ be a mock-Lie algebra, $V$ be a vector space and $\rho : A \to End(V)$ a linear map. Then 
$(V, \rho)$ be a representation of $A$ if and only if 
the direct sum $A \oplus V$ together  with the multiplication  defined by
\begin{equation}
    (x+u)\bullet_{A\oplus V}(y+v)=x\bullet y +\rho(x)v+\rho(y)u,\;\forall x,y\in A,\;\forall u,v\in V, 
\end{equation}
is a mock-Lie algebra. This mock-Lie algebra is called the semi-direct product of $A$ and $V$ and it  is denoted by $A\ltimes_{\rho}V$.
\end{pro}
\begin{defi}
Let $(A,\bullet)$ be a mock-Lie algebra and two representations $(V_1,\rho_1)$ and $(V_2,\rho_2)$. A linear map $\phi:V_1 \rightarrow V_2$ is  said to be a  morphism of representations if 
\begin{equation}
    \rho_2(x)\circ \phi=\phi \circ \rho_1(x),\;\forall x\in A.
\end{equation}
If $\phi$ is bejictive, then $(V_1,\rho_1)$ and $(V_2,\rho_2)$ are equivalent $($isomorphic$)$.
\end{defi}
 To relate matched pairs of mock-Lie algebras to mock-Lie bialgebras and Manin
triples for mock-Lie algebras in the next section, we need the notions of the coadjoint representation, which is the dual
representation of the adjoint representation. In the following we recall these facts. 

Let $(A, \bullet)$ be a mock-Lie algebra and $(V, \rho)$ be a representation of $A$. Let $V^*$ be the dual vector space of $V$. Define the linear map $\rho^*:A \rightarrow End(V^*)$ as
\begin{equation}\label{repdual}
    \langle \rho^*(x)u^*,v\rangle =\langle u^*,\rho(x)v\rangle, \quad \forall x\in A,\;v\in V,\;u^*\in V^*,
\end{equation}
where $\langle\cdot,\cdot\rangle$ is the usual pairing between $V$ and the dual space $V^*$. 
With the above notations, we have the following
\begin{pro}
Let $(V,\rho)$ be a representation of a mock-Lie algebra $(A, \bullet)$. Then $(V^*,\rho^*)$ is a representation of $A$ on $V^*$.
\end{pro}
Consider the case when $V=A$ and 
define the  linear map $L^*:A\rightarrow End(A^*)$ by 
\begin{equation}
    \langle L^*({x})(\xi),y\rangle=\langle \xi,L({x})y \rangle,\quad \forall x,y\in A,\;\xi \in A^*.
\end{equation}
Then we have the following:
\begin{cor}
Let $(A,\bullet)$ be a mock-Lie algebra and $(A,L)$ be the adjoint representation of $A$. Then $(A^*,L^*)$ is a representation of $(A,\bullet)$ on $A^*$ which is called the coadjoint representation.
\end{cor}
If there is a mock-Lie algebra structure on the dual space $A^*$, we denote the left multiplication by $\mathcal{L}$.
\begin{defi}
Let $(A,\bullet)$ be a mock-Lie algebra and $(V,\rho)$ be a representation. A linear map $T:V\rightarrow A$ is called an \textbf{$\mathcal{O}$-operator} associated to$(V,\rho)$ if $T$ satisfies
\begin{equation}
   T(u)\bullet T(v)=T\big(\rho(Tu)v+\rho(Tv)u\big),\quad \forall u,v\in V. 
\end{equation}
In the case $(V,\rho)=(A,L)$, the $\mathcal{O}$-operator $T$ is called a Rota-Baxter operator (of weight zero).
\end{defi}
\begin{defi}
A \textbf{mock-pre-Lie} algebra is a vector space $A$ equipped with a linear map $\c:A\otimes A\rightarrow A$ satisfying the following identity
\begin{equation}\label{premock}
    Aass(x,y,z)=-Aass(y,x,z),\quad \forall x,y,z\in A,
\end{equation}
Recall that  $Aass(x,y,z)=(x\c y)\c z+x\c (y\c z)$. Therefore, Eq. \eqref{premock} is equivalent to
\begin{equation*}
    (x\star y)\c y=(x\c y)\c z-y\c (x\c z),
\end{equation*}
where $x\star y=x\c y+y\c x$, for all $x, y \in A$.
\end{defi}
Note that if $(A,\c)$ is a mock-pre-Lie algebra, then the product given by 
\begin{equation}
    x\star y=x\c y+y\c x,\quad \forall x,y \in A,
\end{equation}
defines a mock-Lie algebra structure, which is called the sub-adjacent mock-Lie algebra of $(A,\c)$, and denoted by $A^{ac}$. Furthermore,  $(A, \c)$
is called the compatible mock-pre-Lie algebra structure on $A^{ac}$.\\
On the other hand, let $\Theta:A \to End(A)$ defined  by  $\Theta(x)y=x\c y,\ \forall x,y\in A$.
Then $(A,\Theta)$ is a representation of the mock-Lie algebra $A^{ac}$.
\begin{pro}
Let $(A,\bullet)$ be a mock-Lie algebra and $(V,\rho)$ be a representation of $A$. If $T$ is an $\mathcal{O}$-operator associated to $(V,\rho)$, then $(V,\c)$ is a  mock-pre-Lie algebra, where 
\begin{equation}
    u\c v=\rho(Tu)v,\quad \forall u,v\in V.
\end{equation}
\end{pro}
\begin{pro}\label{compremock}
Let $(A,\bullet)$ be a mock-Lie algebra. Then there is a compatible  mock-pre-Lie algebra if and only if there exists an invertible $\mathcal{O}$-operator $T : V \rightarrow A$ associated
to a representation $(V, \rho)$. Furthermore, the compatible  mock-pre-Lie structure on $A$ is given by 
\begin{equation}
    x\c y=T\big(\rho(x)T^{-1}(y)\big),\quad \forall x,y\in A.
\end{equation}
\end{pro}
%%%%%%%%%%%%%%%%%%%%%%%%%%%%%%%%%%%%%%%%%%%%%%%%%%%%%
\section{Matched pairs, Manin triples and mock-Lie bialgebras}
%%%%%%%%%%%%%%%%%%%%%%%%%%%%%%%%%%%%%%%%%%%%%%%%%%%%
In this section, we introduce the notions of  Manin triple of a mock-Lie algebra and 
mock-Lie bialgebras. The equivalence between them is interpreted in terms of matched pairs of
mock-Lie algebras.

We first recall the notion of matched pairs of mock-Lie algebras  $($see \cite{Agore2}$)$. Let $(A,\bullet)$ and $(H,\diamond)$ be two mock-Lie algebras. Let $\rho:A\rightarrow End(H)$ and $\mu:H\rightarrow End(A)$ be two linear maps. On the direct sum $A\oplus H$ of the underlying vector spaces, define a linear map $\circ:\otimes^2 (A\oplus H)\rightarrow A\oplus H$ by 
\begin{equation}\label{eqmatch}
    (x+a)\circ(y+b)=x\bullet y +\mu(b)x+\mu(a)y+a\diamond b+\rho(y)a+\rho(x)b,\;\forall x,y\in A,\;a,b\in H.
\end{equation}
\begin{thm}
Let $(A,\bullet)$ and $(H,\diamond)$ be two mock-Lie algebras. Then $(A\oplus H,\circ)$ is a mock-Lie algebra if and only if $(H,\rho)$ and $(A,\mu)$ are representations of $(A,\bullet)$ and $(H,\diamond)$ respectively, and for all $x,y\in A,\;a,b\in H$, the following compatibility conditions are satisfied:
\begin{align}
    &\label{matched1}\rho(x)(a\diamond b)+\rho(x)a \diamond b +a \diamond \rho(x)b+\rho(\mu(a)x)b+\rho(\mu(b)x)a=0,\\
    &\label{matched2}\mu(a)(x\bullet y)+\mu(a)x \bullet y +x \bullet \mu(a)y+\mu(\rho(x)a)y+\mu(\rho(y)a)x=0.
\end{align}

\end{thm}
\begin{defi}
A \textbf{matched pair} of mock-Lie algebras is a quadruple $(A,H;\rho,\mu)$ consisting
of two mock-Lie algebras $(A,\bullet)$ and $(H,\diamond)$, together with representations $\rho:A\rightarrow End(H)$ and $\mu:H\rightarrow End(A)$ respectively, such that the compatibility conditions
\eqref{matched1} and \eqref{matched2} are satisfied.
\end{defi}
\begin{rmk}
We denote the mock-Lie algebra defined by Eq. \eqref{eqmatch} by  $A\bowtie H$. It
is straightforward to show that every  mock-Lie algebra which is a direct sum of the underlying vector spaces
of two  mock-Lie subalgebras can be obtained from a
matched pair of  mock-Lie algebras as above.
\end{rmk}
\begin{defi}\label{invariant}
A bilinear form $\omega$ on a mock-Lie algebra $(A,\bullet)$ is called \textbf{invariant} if it satisfies
\begin{equation}
    \omega (x\bullet y,z)=\omega(x,y\bullet z),\quad \forall x,y,z\in A.
\end{equation}
\end{defi}
\begin{pro}
Let $(A,\bullet)$ be a mock-Lie algebra and $(A,L)$ be the adjoint representation of $A$ on itself. Then $(A,L)$ and $(A^*,L^*)$ are equivalent as representations of the mock-Lie algebra $(A,\bullet)$ if and only if there exists a nondegenerate symmetric invariant bilinear form $\omega$ on $A$.
\end{pro}
\begin{proof}
Suppose that there exists a nondegenerate symmetric invariant bilinear form $\omega$ on $A$.
Since $\omega$ is nondegenerate, there exists a linear isomorphism $\phi:A \rightarrow A^*$ defined by 
\begin{equation*}
    \langle \phi(x),y\rangle=\omega (x,y),\quad \forall x,y\in A.
\end{equation*}
Hence for any $x,y,z\in A$, we have 
\begin{align*}
    &\langle \phi(L(x)(y)),z\rangle =\omega(L(x)(y),z)=\omega(x\bullet y,z)=\omega(y,x\bullet z)\\
    &\quad \quad \quad\quad \;\;\;\; \;\;\;\;\; =\langle \phi(y),x\bullet z\rangle=\langle L^*(x)\phi(y),z\rangle.
\end{align*}
That is, $(A,L)$ and $(A^*,L^*)$ are equivalent. Conversely, by a similar way, we can get the
conclusion.
\end{proof}

\begin{defi}
A \textbf{Manin triple} of mock-Lie algebras is a triple of mock-Lie algebras $(A,A^+,A^-)$ together with a nondegenerate symmetric invariant bilinear  form $\omega$ on $A$ such that the following conditions  are satisfied:\\
$(a)$ $A^+$, $A^-$ are mock-Lie subalgebras of $A$.\\
$(b)$ $A=A^+\oplus A^-$ as vector spaces.\\
$(c)$ $A^+$ and $A^-$ are isotropic with respect to $\omega$, that is, $\omega(x_+,y_+)=\omega(x_-,y_-)=0$, for any $x_+,y_+\in A^+$, $x_-,y_-\in A^-$.
\end{defi}
A homomorphism between two Manin triples of mock-Lie algebras  $(A,A^+,A^-)$ and\\ $(B,B^+,B^-)$ associated to two nondegenerate symmetric invariant bilinear forms $\omega_1$ and $\omega_2$ respectively, is a homomorphism of mock-Lie algebras $f:A \rightarrow B$ such that 
\begin{align*}
      f(A^+)\subset B^+,\quad f(A^-)\subset B^-,\quad \omega_1(x,y)=\omega_2(f(x),f(y)),\ \forall x,y \in A.
\end{align*}
If in addition, $f$ is an isomorphism of vector spaces, then the two Manin triples are called isomorphic.
\begin{defi}\cite{Haliya}
Let $(A,\bullet)$ be a mock-Lie algebra. Suppose that there is a mock-Lie algebra structure $(A^*,\diamond)$ on the dual space $A^*$ of $A$ and there is a mock-Lie algebra structure on the direct sum $A\oplus A^*$ of the underlying vector spaces $A$ and $A^*$ such that $(A,\bullet)$ and $(A^*,\diamond)$ are subalgebras and the natural non-degerenate symmetric bilinear form on $A\oplus A^*$ given by 
\begin{equation}\label{standard bilinear form}
    \omega_d(x+\xi,y+\eta):=\langle x,\eta\rangle +\langle \xi,y\rangle,\quad  \forall x,y\in A,\;\xi ,\eta\in A^*,
\end{equation}
is invariant, then $(A\oplus A^*,A,A^*)$ is called a standard Manin triple of mock-Lie algebra associated to standard bilinear form $\omega_d$.
\end{defi}
Obviously, a standard Manin triple of mock-Lie algebras is a Manin triple of mock-Lie 
algebras. Conversely, we have
\begin{pro}
Every Manin triple of mock-Lie algebras is isomorphic to a standard one.
\end{pro}
\begin{proof}
Since $A^+$ and $A^-$ are isotropic under the nondegenerate invariant bilinear form $\omega$ on $A^+\oplus A^-$, then in this case $A^-$ and $(A^+)^*$ are identified by $\omega$ and the mock-Lie algebra structure on $A^-$ is transferred to $(A^+)^*$.  Hence the mock-Lie algebra structure on $A^+\oplus A^-$ is transferred to $A^+\oplus (A^+)^*$. Transfer the nondegenrate bilinear form $\omega$ to $A^+\oplus (A^+)^*$, we obtain the standard bilinear form given by \eqref{standard bilinear form}. Thus, $(A,A^+,A^-)$ is isomorphic to the stansadrd Manin triple  $(A\oplus A^*,A,A^*)$.
\end{proof}
\begin{pro} \cite{Haliya}\label{matchmanin}
Let $(A,\bullet)$ be a mock-Lie algebra. Suppose that there is a mock-Lie algebra structure $(A^*,\diamond)$ on $A^*$. Then there exists a mock-Lie algebra sructure on the vector space $A\oplus A^*$ such that $(A\oplus A^*,A,A^*)$ is a standard Manin triple of mock-Lie algebras with respect to $\omega_d$ defined by \eqref{standard bilinear form} if and only if $(A,A^*;L^*,\mathcal{L}^*)$  is a matched pair of mock-Lie algebras. Here $\mathcal{L}^*$ is the coadjoint representation of the mock-Lie algebra $(A^*,\diamond)$.
\end{pro}

\begin{pro}\label{caracmatched}
Let $(A,\bullet)$ be a mock-Lie algebra. Suppose that there is a mock-Lie algebra structure  $(A^*,\diamond)$ on $A^*$. Then $(A,A^*;L^*,\mathcal{L}^*)$  is a matched pair of mock-Lie algebras if and only if for any $x,y\in A,\xi\in A^*$, we have
\begin{equation}\label{condmatched}
    \mathcal{L^*}(\xi)(x\bullet y)+(\mathcal{L^*}(\xi)(x))\bullet y+x\bullet (\mathcal{L^*}(\xi)(y))+ \mathcal{L^*}(L^*(x)(\xi))(y)+\mathcal{L^*}(L^*(y)(\xi))(x)=0.
\end{equation}
\end{pro}
\begin{proof}
Obiviously, Eq.\eqref{condmatched} is exactly Eq. \eqref{matched2} in the case $\rho=L^*$, $\mu=\mathcal{L^*}$. In addition, for any $x,y\in A$, $\xi, \eta \in A^*$, we have
\begin{align*}
    &\langle \mathcal{L^*}(\xi)(x\bullet y),\eta\rangle =\langle x\bullet y,\mathcal{L}(\xi)(\eta)\rangle=\langle L(x)(y),\xi \diamond \eta\rangle=\langle y,L^*(x)(\xi \diamond \eta)\rangle;\\
    &\langle (\mathcal{L^*}(\xi)(x))\bullet y,\eta\rangle =\langle L(\mathcal{L^*}(\xi)(x))(y),\eta\rangle =\langle y,L^*(\mathcal{L^*}(\xi)(x))(\eta)\rangle;\\
    &\langle x\bullet (\mathcal{L^*}(\xi)(y)),\eta \rangle=\langle L(x)(\mathcal{L^*}(\xi)(y)),\eta\rangle=\langle \mathcal{L^*}(\xi)(y),L^*(x)(\eta)\rangle=\langle y,\mathcal{L}(\xi)(L^*(x)(\eta))\rangle\\
    &\quad \quad \quad \quad \quad \quad\quad \;\; =\langle y,\xi \diamond (L^*(x)(\eta))\rangle;\\
    &\langle \mathcal{L^*}(L^*(x)(\xi))(y),\eta\rangle=\langle y,\mathcal{L}(L^*(x)(\xi))(\eta)\rangle=\langle y,(L^*(x)(\xi))\diamond \eta \rangle;\\
    &\langle \mathcal{L^*}(L^*(y)(\xi))(x),\eta\rangle =\langle x,\mathcal{L}(L^*(y)(\xi))(\eta)\rangle=\langle x,\eta \diamond (L^*(y)(\xi))\rangle=\langle x,\mathcal{L}(\eta)((L^*(y)(\xi))\rangle\\
     &\quad \quad \quad \quad \quad \quad\quad \quad\;\; =\langle\mathcal{L^*}(\eta)(x),L^*(y)(\xi)\rangle=\langle L(\mathcal{L^*}(\eta)(x))(y),\xi\rangle=\langle y,L^*(\mathcal{L^*}(\eta)(x))(\xi)\rangle.
\end{align*}
Then Eq. \eqref{matched1} holds if and only if Eq. \eqref{matched2} holds. Therefore the conclusion holds.
\end{proof}

\begin{thm}\label{bialg}
Let $(A,\bullet)$ be a mock-Lie algebra. Suppose that there is a mock-Lie algebra structure $"\diamond"$ on its dual space $A^*$ given by a linear map $\Delta^*:A^*\otimes A^*\rightarrow A^*$, that is, $\xi\diamond \eta=\Delta^*(\xi \otimes \eta)$, for any $\xi,\eta\in A^*$. Then $(A,A^*;L^*,\mathcal{L}^*)$  is a matched pair of mock-Lie algebras if and only if $\Delta:A \rightarrow A\otimes A$ satisfies the following condition:
\begin{equation}\label{bialgebra}
    \Delta(x\bullet y)=-\big(L(x)\otimes id+id \otimes L(x)\big)\Delta(y)-\big(L(y)\otimes id+id \otimes L(y)\big)\Delta(x),
\end{equation}
for any $x,y\in A$.
\end{thm}
\begin{proof}
Using Proposition \ref{caracmatched}, we can prove that Eq. \eqref{bialgebra} is equivalent to Eq. \eqref{condmatched}. In fact, for any $x,y\in A$, $\xi,\eta\in A^*$, we have 
\begin{align*}
    &\langle \mathcal{L^*}(\xi)(x\bullet y),\eta\rangle=\langle x\bullet y,\xi \diamond \eta\rangle =\langle x\bullet y,\Delta^*(\xi\otimes \eta)\rangle=\langle \Delta(x\bullet y),\xi \otimes \eta \rangle;\\
    &\langle (\mathcal{L^*}(\xi)(x))\bullet y,\eta \rangle =\langle L(y)(\mathcal{L^*}(\xi)(x)),\eta\rangle=\langle \mathcal{L^*}(\xi)(x),L^*(y)(\eta)\rangle =\langle x,\mathcal{L}(\xi)(L^*(y)(\eta))\rangle\\
    &\quad \quad \quad \quad \quad \quad \quad \;\;=\langle x,\xi \diamond (L^*(y)(\eta))\rangle =\langle (id \otimes L(y))\Delta(x),\xi \otimes \eta \rangle;\\
    &\langle x\bullet (\mathcal{L^*}(\xi)(y)),\eta\rangle =\langle L(x)(\mathcal{L^*}(\xi)(y)),\eta \rangle=\langle \mathcal{L^*}(\xi)(y),L^*(x)(\eta)\rangle =\langle y,\mathcal{L}(\xi)(L^*(x)(\eta))\rangle\\
    &\quad \quad \quad \quad \quad \quad \quad \;\;=\langle y,\xi \diamond (L^*(x)(\eta))\rangle=\langle (id \otimes L(x))\Delta(y),\xi \otimes \eta \rangle;\\
    &\langle \mathcal{L^*}(L^*(x)(\xi))(y),\eta\rangle =\langle y,(L^*(x)(\xi))\diamond \eta \rangle=\langle (L(x)\otimes id )\Delta(y),\xi \otimes \eta \rangle;\\
    &\langle \mathcal{L^*}(L^*(y)(\xi))(x),\eta \rangle =\langle x,(L^*(y)(\xi))\diamond \eta\rangle=\langle (L(y)\otimes id )\Delta(x),\xi \otimes \eta\rangle. 
    \end{align*}
    Then Eq. \eqref{condmatched} is equivalent to Eq. \eqref{bialgebra}. Hence the conclusion holds.
\end{proof}
\begin{rmk}\label{gamma}
From the symmetry of the mock-Lie algebras $(A,\bullet)$ and $(A^*,\diamond)$ in the standard Manin triple of mock-Lie algebras with respect to $\omega_d$, we also can consider a linear map $\gamma: A^*\rightarrow A^*\otimes A^*$ such that $\gamma^*:A \otimes A\rightarrow A$ gives the mock-Lie algebra structure $"\bullet"$ on $A$. It is straightforward to show that $\Delta$ satisfies Eq. \eqref{bialgebra} if and only if $\gamma$ satisfies 
\begin{equation}\label{bialgebra2}
    \gamma(\xi\diamond \eta)=-\big(\mathcal{L}(\xi)\otimes id+id \otimes \mathcal{L}(\xi)\big)\gamma(\eta)-\big(\mathcal{L}(\eta)\otimes id+id \otimes \mathcal{L}(\eta)\big)\gamma(\xi),
\end{equation}
for any $\xi,\eta\in A^*$.
\end{rmk}
\begin{defi}\label{definition1}
Let $(A,\bullet)$ be a mock-Lie algebra. A \textbf{mock-Lie bialgebra} structure on $A$ is a symmetric linear map $\Delta:A \rightarrow A\otimes A$ such that 
\begin{enumerate}
    \item $\Delta^*:A^*\otimes A^*\rightarrow A^*$ defines a mock-Lie algebra structure on $A^*$;
    \item $\Delta$ satifies Eq. \eqref{bialgebra}, called the compatibility condition.
\end{enumerate}
We denote it by $(A,\Delta)$ or $(A,A^*)$.
\end{defi}

We can unwrap the compatibility condition Eq. \eqref{bialgebra} as 
\begin{equation}
    \Delta(x\bullet y)=-(x\bullet y_1)\otimes y_2-y_1\otimes (x\bullet y_2)-(y\bullet x_1)\otimes x_2-x_1\otimes (y\bullet x_2).
\end{equation}
\begin{rmk}
The compatibility condition Eq. \eqref{bialgebra} is, in fact, a cocycle condition in the zigzag cohomology of mock-Lie algebra  introduced in \cite{Benayadi}. Indeed, we can regard $A^{\otimes2}$ as an $A$-module via the adjoint action \eqref{adjointaction}:
\begin{equation}
    x\c(y_1\otimes y_2)=L(x)(y_1\otimes y_2)=(x\bullet y_1)\otimes y_2+y_1\otimes (x\bullet y_2),
\end{equation}
for $x\in A$ and $y_1\otimes y_2\in A^{\otimes2}$. Then we can think of the linear map $\Delta:A \rightarrow A^{\otimes2}$ as a $1$-cochain. Then the differential on $\Delta$ is given by 
\begin{align*}
    &d^1\Delta(x,y)=\Delta(x\bullet y)+x\c \Delta(y)+y\c \Delta(x)\\
    &\quad \quad \quad \;\;\;\; =\Delta(x\bullet y)+L(x)(\Delta(y))+L(y)(\Delta(x)).
\end{align*}
Therefore, Eq. \eqref{bialgebra} says exactly that $\Delta\in C^1(A,A^{\otimes2})$ is a  $1$-cocycle.
\end{rmk}
\begin{ex}
Let $(A,A^*)$ be a mock-Lie bialgebra on a mock-Lie algebra $(A,\bullet)$. Then $(A^*,\gamma)(or (A^*,A))$ is a mock-Lie bialgebra on the mock-Lie algebra $(A^*,\diamond)$, where $\gamma$ is given in Remark \ref{gamma}.
\end{ex}

\begin{defi}
Let $(A_1,A^*_1)$ and $(A_2,A^*_2)$ be two mock-Lie bialgebras. A linear map $\psi:A_1\rightarrow A_2$ is a \textbf{homomorphism} of mock-Lie bialgebras if $\psi $ satifies, for any $x,y\in A_1$
\begin{equation}
    \psi(x\bullet_1 y)=\psi(x)\bullet_2 \psi(y),\quad \quad (\psi\otimes \psi)\circ \Delta_1=\Delta_2\circ \psi.
\end{equation}
\end{defi}
Now, combining Proposition \ref{matchmanin} and Theorem \ref{bialg}, we have the following conclusion.
\begin{thm}
Let $(A,\bullet)$ be a mock-Lie algebra. Suppose that there is a mock-Lie algebra
structure on $A^*$ denoted by $"\diamond"$ which is defined as a linear map $\Delta : A \rightarrow A\otimes A$. Then
the following conditions are equivalent:
\begin{enumerate}
    \item $(A\oplus A^*,A,A^*)$ is a standard  Manin triple of mock-Lie algebras with respect to $\omega_d$ defined by Eq. \eqref{standard bilinear form}.
    \item $(A,A^*;L^*,\mathcal{L^*})$ is a matched pair of mock-Lie algebras.
    \item $(A,A^*)$ is a mock-Lie bialgebra.
\end{enumerate}
\end{thm}

%%%%%%%%%%%%%%%%%%%%%%%%%%%%%%%%%%%%%%%%%%%%%%%%%%%
\section{Coboundary mock-Lie bialgebras and the mock-Lie Yang-Baxter equation  }
%%%%%%%%%%%%%%%%%%%%%%%%%%%%%%%%%%%%%%%%%%%%%%%%%%
In this section, we consider a special class of mock-Lie bialgebras called coboundary mock-Lie bialgebras and introduce the notion of mock-Lie Yang-Baxter equation.
\begin{defi}
A mock-Lie bialgebra $(A,A^*)$ is called \textbf{coboundary} if there exists an element $r\in A\otimes A$ such that for any $x\in A$, 
\begin{equation}\label{coboundary}
    \Delta(x)=\big(L(x)\otimes id -id \otimes L(x)\big)r.
\end{equation}
\end{defi}
\begin{lem}
Let $(A,\bullet)$ be a mock-Lie algebra and $r\in A\otimes A$. Suppose that the linear map $\Delta:A\rightarrow A\otimes A$ is defined by Eq. \eqref{coboundary}. Then $\Delta$ satisfies the compatibility condition given by Eq. \eqref{bialgebra}. 
\end{lem}
\begin{proof}
Let $r=r_1\otimes r_2\in A\otimes A$. Using the commutativity and Jacobi identity for mock-Lie algebras. Then for any $x,y\in A$, we have 
\begin{align*}
    &-\big(L(x)\otimes id+id \otimes L(x)\big)\Delta(y)-\big(L(y)\otimes id+id \otimes L(y)\big)\Delta(x)\\
    =&-\big(L(x)\otimes id+id\otimes L(x)\big)\big(L(y)\otimes id -id\otimes L(y)\big)r\\
    &-\big(L(y)\otimes id+id\otimes L(y)\big)\big(L(x)\otimes id -id\otimes L(x)\big)r\\
    =&-\big(L(x)\otimes id+id\otimes L(x)\big)\big((y\bullet r_1)\otimes r_2-r_1\otimes (y\bullet r_2)\big)\\
    & -\big(L(y)\otimes id+id\otimes L(y)\big)\big((x\bullet r_1)\otimes r_2-r_1\otimes (x\bullet r_2)\big)\\
    =&-\Big(x\bullet (y\bullet r_1)\otimes r_2-(x\bullet r_1)\otimes (y\bullet r_2)+(y\bullet r_1)\otimes (x\bullet r_2)-r_1\otimes (x\bullet (y\bullet r_2))\Big)\\
    &-\Big(y\bullet (x\bullet r_1)\otimes r_2-(y\bullet r_1)\otimes (x\bullet r_2)+(x\bullet r_1)\otimes (y\bullet r_2)-r_1\otimes (y\bullet (x\bullet r_2))\Big)\\
    =&\big(-x\bullet (y\bullet r_1)-y\bullet (x\bullet r_1)\big)\otimes r_2+r_1\otimes \big(x\bullet (y\bullet r_2)+y\bullet (x\bullet r_2)\big)\\
    =&\big((x\bullet y)\bullet r_1\big)\otimes r_2-r_1\otimes \big((x\bullet y)\bullet r_2\big)
    =\big(L(x\bullet y)\otimes id -id \otimes L(x\bullet y)\big)r\\
    =&\Delta(x\bullet y).
\end{align*}
Hence the proof.
\end{proof}
Let $\Delta:A\rightarrow A\otimes A$ be a linear map and $\sigma:A^{\otimes3}\rightarrow A^{\otimes3}$ be defined as $\sigma(x\otimes y\otimes z)=y\otimes z \otimes x$, for any $x,y,z\in A$. Let $E_{\Delta}:A \rightarrow A^{\otimes3}$ be a linear map given by
\begin{equation}
   E_{\Delta}(x)=(id+\sigma+\sigma^2)\big((id\otimes \Delta)\Delta(x)\big). 
\end{equation}
\begin{lem}\label{lemme1}
Let $A$ be a vector space and $\Delta:A \rightarrow A \otimes A$ be a linear map. Then the product  $"\diamond"$ in $A^*$ given by $\Delta^*:A^*\otimes A^*\rightarrow A^*$ satisfies the Jacobi identity if and only if $E_{\Delta}=0$.
\end{lem}
\begin{proof}
For any $\xi,\eta\in A^*$, $x\in A$, we have
\begin{equation}\label{Delta}
    \langle \xi\diamond \eta,x\rangle =\langle \Delta^*(\xi \otimes \eta),x\rangle =\langle \xi \otimes \eta,\Delta(x)\rangle.
\end{equation}
Threrefore, for any $\xi,\eta,\nu\in A^*$ and $x\in A$, the Jacobi identity satisfies
\begin{align*}
    &\langle J(\xi ,\eta, \nu),x\rangle\\
    =&\langle \Delta^*(id\otimes \Delta^*)(\xi\otimes \eta\otimes \nu)+\Delta^*(id\otimes \Delta^*)(\eta\otimes \nu\otimes \xi)+\Delta^*(id\otimes \Delta^*)(\nu\otimes \xi\otimes \eta),x\rangle\\
    =&\langle \Delta^*(id\otimes \Delta^*)\big(id+\sigma+\sigma^2\big)(\xi\otimes \eta\otimes \nu),x\rangle\\
    =&\langle\xi\otimes \eta\otimes \nu , \big(id+\sigma+\sigma^2\big)((id\otimes \Delta)\Delta)(x) \rangle.
\end{align*}
Therefore $J(\xi ,\eta, \nu)=0$, for any $\xi,\eta,\nu\in A^*$ if and only if $E_{\Delta}=0$.
\end{proof}
Let $(A,\bullet)$ be a mock-Lie algebra and $r=\sum_{i}a_i\otimes b_i\in A\otimes A$. Set 
\begin{equation}
    r_{12}=\sum_{i}a_i\otimes b_i\otimes 1,\quad r_{13}=\sum_{i}a_i\otimes 1 \otimes b_i,\quad r_{23}=\sum_{i}1\otimes a_i\otimes b_i,
\end{equation}
where $1$ is a unit element if $(A,\bullet)$  is unital or a symbol playing a similar role of the unit for the non-unital cases. The operation between two $r_{ij}$ is in obvious way. For example,
{\small\begin{align}
    &r_{12}\bullet r_{13}=\sum_{ij}a_i\bullet a_j\otimes b_i\otimes b_j,\ \ r_{13}\bullet r_{23}=\sum_{ij}a_i\otimes a_j\otimes b_i\bullet b_j,\ \  r_{23}\bullet r_{12}=\sum_{ij}a_j\otimes a_i\bullet b_j\otimes b_i.
\end{align}}

Note that the above elements are independents of the existence of the unit.
A tensor $r\in A\otimes A$ is called symmetric (resp. skew-symmetric) if $r=\tau(r)$ ( resp. $r=-\tau(r)$). On the other hand, any $r\in A\otimes A$ can be identified as a linear map from the dual space $A^*$ to $A$ in the following way:
\begin{equation}
    \langle \xi,r(\eta)\rangle =\langle \xi \otimes \eta,r\rangle, \quad \forall \xi,\eta\in A^*.
\end{equation}
The tensor $r\in A\otimes A$ is called nondegenerate if the above induced linear map is invertible.
\begin{pro}\label{proposition1}
Let $(A,\bullet)$ be a mock-Lie algebra. Define a linear map $\Delta:A\rightarrow A\otimes A$ by Eq. \eqref{coboundary} with some $r\in A\otimes A$ satisfying 
\begin{equation}\label{SymmCobracket}
    \big(L(x)\otimes id -id \otimes L(x)\big)\big(r+\tau(r)\big)=0,
\end{equation}
for all $x\in A$. Then
\begin{equation}
    E_{\Delta}(x)+Q(x)[[r,r]]=0,
\end{equation}
where \begin{equation}\label{mockYBE}
    [[r,r]]=r_{12}\bullet r_{13}+r_{13}\bullet r_{23}-r_{12}\bullet r_{23},
\end{equation}
and $Q(x)=\big(L(x)\otimes id \otimes id +id\otimes L(x)\otimes id +id\otimes id \otimes L(x)\big)$ for any $x\in A$.
\end{pro}
\begin{proof}
Let $r=\sum_ia_i\otimes b_i$, the condition \eqref{SymmCobracket} is equivalent to
\begin{equation}\label{cond}
  \sum_i(x\bullet a_i)\otimes b_i-a_i\otimes (x\bullet b_i)+(x\bullet b_i)\otimes a_i-b_i\otimes (x\bullet a_i)=0.  
\end{equation}

 Note that $E_{\Delta}(x)$ is the sum of twelve terms and that $Q(x)[[r,r]]$ is a sum of nine terms, but two terms appear in both sums up to sign and hence are canceled. Thus $E_{\Delta}(x)+Q(x)[[r,r]]$ is a sum of seventeen terms. After rearranging the terms suitably, we obtain 
\begin{align*}
    &E_{\Delta}(x)+Q(x)[[r,r]]\\
    =&\sum_{i,j}\big\{-(x\bullet b_i)\bullet a_j\otimes b_j\otimes a_i+x\bullet (a_i\bullet a_j)\otimes b_i\otimes b_j+(x\bullet b_i)\bullet b_j\otimes a_i\otimes a_j\\
    &-(b_i\bullet b_j)\otimes (x\bullet a_i)\otimes a_j+(a_i\bullet a_j)\otimes (x\bullet b_i)\otimes b_j+(b_i\bullet a_j)\otimes b_j\otimes (x\bullet a_i)\\
    &+(a_i\bullet a_j)\otimes b_i\otimes (x\bullet b_j)-a_i\otimes (x\bullet b_i)\bullet a_j\otimes b_j+a_i\otimes a_j\otimes (x\bullet b_i)\bullet b_j\\
    &+b_j\otimes (x\bullet a_i)\otimes (b_i\otimes a_j)-b_j\otimes a_i\otimes (x\bullet b_i)\bullet a_j-a_j\otimes (b_i\bullet b_j)\otimes (x\bullet a_i)\\
    &+a_j\otimes (x\bullet b_i)\bullet b_j\otimes a_i-a_i\otimes x\bullet (b_i\bullet a_j)\otimes b_j-a_i\otimes (b_i\bullet a_j)\otimes (x\bullet b_j)\\
    &+a_i\otimes (x\bullet a_j)\otimes (b_i\bullet b_j)+a_i\otimes a_j\otimes x\bullet (b_i\bullet b_j)\big\}.
\end{align*}
Interchanging the indices $i$ and $j$ in the first term and using the Jacobi identity in $A$, the
first term becomes
\begin{equation*}
   \sum_{i,j} x\bullet (b_j\bullet a_i)\otimes b_i\otimes a_j+b_j\bullet (a_i\bullet x)\otimes b_i\otimes a_j.
\end{equation*}
Using the Eq. \eqref{cond}, the sum of $b_j\bullet (a_i\bullet x)\otimes b_i\otimes a_j$ and the third and fourth terms is
\begin{align*}
    &\sum_{i,j} \big(L(b_j)\otimes id\big)\big((a_i\bullet x)\otimes b_i+(x\bullet b_i)\otimes a_i -b_i\otimes (x\bullet a_i)\big)\otimes a_j\\&\sum_{j} \big(L(b_j)\otimes id\big)\sum_{i}\big((a_i\bullet x)\otimes b_i+(x\bullet b_i)\otimes a_i -b_i\otimes (x\bullet a_i)\big)\otimes a_j\\
    =&\sum_{i,j}\big(L(b_j)\otimes id\big)\big(a_i\otimes (x\bullet b_i)\big)\otimes a_j\\
    =&\sum_{i,j}(a_i\bullet b_j)\otimes (x\bullet b_i)\otimes a_j.
\end{align*}
Similarly, the sum of $(a_i\bullet b_j)\otimes (x\bullet b_i)\otimes a_j$ and the fifth term becomes
\begin{equation*}
   \sum_{i,j} b_j\otimes (x\bullet b_i)\otimes (a_i\bullet a_j)+a_j\otimes (x\bullet b_i)\otimes (a_i\bullet b_j),
\end{equation*}
and the sum of the sixth and seventh terms is
\begin{equation*}
   \sum_{i,j} a_j\otimes b_i\otimes x\bullet (a_i\bullet b_j)+b_j\otimes b_i\otimes x\bullet (a_i\bullet a_j).
\end{equation*}
Finally, the sum of $x\bullet (b_j\bullet a_i)\otimes b_i\otimes a_j$ and the second term in the sum of the
expression of $E_{\Delta}(x)+Q(x)[[r,r]]$ becomes
\begin{equation*}
  \sum_{i,j}  b_j\otimes b_i\otimes a_i\bullet (x\bullet a_j)+a_j\otimes b_i\otimes a_i\bullet (x\bullet b_j).
\end{equation*}
Inserting these results, we find that the expression of $E_{\Delta}(x)+Q(x)[[r,r]]$ can be written in the form $\sum_i\big(a_i\otimes U_i+b_i\otimes V_i\big)$: In fact, 
\begin{align*}
    U_i=&\sum_{j}\big\{-(b_j\bullet b_i)\otimes (x\bullet a_j)-(b_i\bullet a_j)\otimes (x\bullet b_j)+b_j\otimes x\bullet (a_j\bullet b_i)\\
    &+a_j\otimes x\bullet (b_i\bullet b_j)+(x\bullet b_j)\otimes (a_j\bullet b_i)+(x\bullet a_j)\otimes (b_i\bullet b_j)\\
    &-x\bullet (b_i\bullet a_j)\otimes b_j+(x\bullet b_j)\bullet b_i\otimes a_j-(x\bullet b_i)\bullet a_j\otimes b_j\\
    &+a_j\otimes (x\bullet b_i)\bullet b_j+b_j\otimes (x\bullet b_i)\bullet a_j\big\}.
\end{align*}
On the right-hand side, the sum of the first four terms is zero by Eq. \eqref{cond}, and the sum of the next
three terms becomes 
\begin{equation*}
    x\bullet (b_i\bullet b_j)\otimes a_j.
\end{equation*}
By the Jacobi identity in $A$, the sum of $x\bullet (b_i\bullet b_j)\otimes a_j$ and the eighth term is
\begin{equation*}
    -b_j\bullet (x\bullet b_i)\otimes a_j.
\end{equation*}
Finally, the sum of $-b_j\bullet (x\bullet b_i)\otimes a_j$ and the last three terms becomes
\begin{equation*}
\sum_{j}  -b_j\bullet (x\bullet b_i)\otimes a_j  -(x\bullet b_i)\bullet a_j\otimes b_j+a_j\otimes (x\bullet b_i)\bullet b_j+b_j\otimes (x\bullet b_i)\bullet a_j=0,
\end{equation*}
if we replace $x$ in Eq. \eqref{cond} by $x\bullet b_i$: Hence, we get $U_i= 0$. Similarly, we can
proves that
\begin{align*}
    V_i=&\sum_{j}\big\{(x\bullet b_j)\otimes (a_j\bullet a_i)+b_j\otimes x\bullet (a_j\bullet a_i)+b_j\otimes a_j\bullet (x\bullet a_i)\\
    &\quad\quad\quad\quad+(x\bullet a_j)\otimes (b_j\bullet a_i)-a_j\otimes (x\bullet b_j)\bullet a_i\big\}\\=&0.
\end{align*}
Hence the conclusion holds.
\end{proof}
Using  the above  discussion, we have the following result.
\begin{thm}\label{thmcobou}
Let $(A,\bullet)$ be a mock-Lie algebra and $r\in A\otimes A$. Define a bilinear map $\diamond:A^*\otimes A^*\rightarrow A^*$ by 
\begin{equation*}
    \langle \xi\diamond \eta,x\rangle =\langle \Delta^*(\xi \otimes \eta),x\rangle =\langle \xi \otimes \eta,\Delta(x)\rangle,
\end{equation*}
where $\Delta$ is defined by Eq. \eqref{coboundary}. Then $(A^*,\diamond)$ is a mock-Lie algebra if and only
if the following conditions are satisfied:
\begin{align*}
    &(i)\quad \big(L(x)\otimes id -id \otimes L(x)\big)\big(r+\tau(r)\big)=0,\\
    &(ii)\quad Q(x)[[r,r]]=0,
\end{align*}
for all $x\in A$. Under these conditions, $(A,A^*)$ is a coboundary mock-Lie bialgebra.
\end{thm}
\begin{proof}
The bracket $\diamond$ is determined by the cobracket $\Delta(x)=\big(L(x)\otimes id -id \otimes L(x)\big)r$. Hence $(A^*,\diamond)$
is a  mock-Lie algebra if
and only if $\diamond$ is symmetric and satisfies the Jacobi identity.

For any $x\in A,\;\xi,\eta\in A^*$, we have 
\begin{align*}
    \langle\xi \diamond \eta-  \eta\diamond\xi,x\rangle &=\langle \Delta^*(\xi \otimes \eta)- \Delta^*( \eta\otimes\xi ),x\rangle=\langle \xi\otimes \eta,\Delta(x)-\tau\circ \Delta(x)\rangle\\
    & =\langle \xi \otimes \eta, \big(L(x)\otimes id -id \otimes L(x)\big)r-\tau\circ( \big(L(x)\otimes id -id \otimes L(x)\big)r)\rangle\\ &=\langle \xi \otimes \eta,L(x)r_1\otimes r_2-r_1\otimes L(x)r_2-r_2\otimes L(x)r_1+L(x)r_2\otimes r_1\rangle\\
    &=\langle \xi \otimes \eta,\big(L(x)\otimes id -id \otimes L(x)\big)\big(r+\tau(r)\big)\rangle.
\end{align*}
Then  $r$ satisfies $(i)$ if and only if  $\diamond$ is symmetric. The
proof that $(ii)$ is equivalent to the condition that $\diamond$ satisfies the Jacobi identity which follows
from Lemma \ref{lemme1} and Proposition \ref{proposition1}. Since $\Delta(x)=\big(L(x)\otimes id -id \otimes L(x)\big)r$, the compatibility conditions
for a mock-Lie bialgebra in Definition \ref{definition1} hold naturally. Therefore the conclusion follows.
\end{proof}
\begin{rmk}
An easy way to satisfy conditions $(i)$ and $(ii)$ in Theorem \ref{thmcobou} is to assume that $r$ is
skew-symmetric and 
\begin{equation}\label{mockyb}
    [[r,r]]=0
\end{equation}
respectively. Eq. \eqref{mockyb} is the \textbf{mock-Lie Yang-Baxter equation} in the mock-Lie algebra $(A,\bullet)$. A quasitriangular mock-Lie bialgebra
is a coboundary mock-Lie bialgebra, in which $r$ is a solution of the mock-Lie Yang-
Baxter equation. A triangular mock-Lie bialgebra is a coboundary mock-Lie bialgebra, in which $r$
is a skew-symmetric solution of the mock-Lie Yang-Baxter equation.
\end{rmk}
A direct application of Theorem \ref{thmcobou} is given as follows.
\begin{thm}\label{double}
Let $(A,A^*)$ be a mock-Lie bialgebra. Then there is a canonical coboundary
mock-Lie  bialgebra structure on $A\oplus A^*$ such that both  $i_1 : A \rightarrow A\oplus A^*$ and
$i_2 : A^* \rightarrow A \oplus A^*$
into the two summands are homomorphisms of mock-Lie bialgebras. Here the mock-Lie bialgebra structure on $A$ is $(A,-\Delta_A)$
, where $\Delta_A$ is given by Eq. \eqref{coboundary}.
\end{thm}
\begin{proof}%$\langle\eta,x\rangle=\langle\eta,id(x)\rangle=\langle\eta\otimes x,r\rangle$
Let $r\in A\otimes A^*\subset (A\oplus A^*)\otimes (A\oplus A^*)$ correspond to the identity map $id : A \rightarrow A$. Let $\{e_1,\cdots,e_n\}$ be a basis of $A$ and $\{f_1,\cdots,f_n\}$ be its dual basis. Then $r=\sum_ie_i\otimes f_i$. Suppose that the mock-Lie algebra structure $"\circ_{D(A)}"$ on $A\oplus A^*$ is given by $D(A)=A\bowtie A^*$. Then  by Eq. \eqref{eqmatch}, we have 
\begin{equation*}
    x\circ_{D(A)} y=x\bullet y,\quad a^*\circ_{D(A)} b^*=a^*\diamond b^*,\quad x\circ_{D(A)} a^*=L^*(x)a^*+\mathcal{L^*}(a^*)x,
\end{equation*}
for any $x,y\in A$, $a^*,b^*\in A^*$.  Next we prove that $r$ satisfies the two conditions in Theorem \ref{thmcobou}.
If so, then
\begin{equation*}
    \Delta_{D(A)}(u)=(L_{\circ_{D(A)}}(u)\otimes id_{D(A)}-id_{D(A)}\otimes L_{\circ_{D(A)}}(u))r,\quad \forall u\in D(A)
\end{equation*}
can induce a coboundary mock-Lie bialgebra structure on $D(A)$. Since 
\begin{equation*}
    \langle \sum_i e_i\otimes f_i,f_s\otimes e_t\rangle=\langle e_t,f_s\rangle,
\end{equation*}
we have 
\small{\begin{align*}
    &\langle [[r,r]]_{D(A)},(e_s+f_t)\otimes (e_k+f_l)\otimes (e_p+f_q)\rangle\\
    =&\sum_{ij}\langle e_i\circ_{D(A)}e_j\otimes f_i\otimes f_j-e_i\otimes f_i\circ_{D(A)}e_j\otimes f_j+e_i\otimes e_j\otimes f_i\circ_{D(A)}f_j,(e_s+f_t)\otimes (e_k+f_l)\otimes (e_p+f_q)\rangle\\
    =&\sum_{ij}\langle e_i\bullet e_j\otimes f_i\otimes f_j-e_i\otimes (\mathcal{L}^*(f_i)e_j+L^*(e_j)f_i)\otimes f_j+e_i\otimes e_j\otimes f_i\diamond f_j,(e_s+f_t)\otimes (e_k+f_l)\otimes (e_p+f_q)\rangle\\
    =& \sum_{ij}\Big(\langle e_i\bullet e_j,f_t\rangle\langle f_i,e_k\rangle \langle f_j,e_p\rangle -\langle e_i,f_t\rangle \langle e_j,f_i\diamond f_l\rangle \langle f_j,e_p\rangle -\langle e_i,f_t\rangle \langle f_i,e_j\bullet e_k\rangle \langle f_j,e_p\rangle \\
    &+\langle e_i,f_t\rangle \langle e_j,f_l\rangle \langle f_i\diamond f_j,e_p\rangle \Big)\\
    =& \langle e_k\bullet e_p,f_t\rangle -\langle e_p,f_t\diamond f_l\rangle -\langle f_t,e_p\bullet e_k\rangle +\langle f_t\diamond f_l,e_p\rangle\\
    =&0,
\end{align*}}
we get $[[r,r]]_{D(A)}=0$. Similarly, we prove that $$\big(L_{\circ_{D(A)}}(u)\otimes id_{D(A)} -id_{D(A)} \otimes L_{\circ_{D(A)}}(u)\big)\big(r+\tau(r)\big)=0,$$ for all $u\in D(A)$. Hence there is a coboundary mock-Lie bialgebra structure on $D(A)$ by Theorem \ref{thmcobou}.
For $e_i \in A$, we have
\begin{align*}
    &\Delta_{D(A)}(e_i)=\sum_j\big\{e_i\bullet e_j\otimes f_j -e_j\otimes e_i\circ _{D(A)}f_j\big\}\\
    &\quad \quad \quad\quad \; =\sum_{j}\big\{e_i\bullet e_j\otimes f_j-e_j\otimes\big(L^*(e_i)f_j+\mathcal{L}^*(f_j)e_i\big)\big\}\\
    &\quad\quad \quad\quad \;=\sum_{j,m}\big\{e_i\bullet e_j\otimes f_j-\langle f_j,e_i\bullet e_m\rangle e_j\otimes f_m-\langle f_j\diamond f_m,e_i\rangle e_j\otimes e_m\big\}\\
    &\quad\quad \quad\quad \;=-\sum_{j,m}\langle f_j\diamond f_m,e_i\rangle e_j\otimes e_m\\
    &\quad \quad \quad\quad \;=-\Delta_A(e_i).
\end{align*}
Therefore  $i_1 : A \rightarrow A \oplus A^*$
is a homomorphism of mock-Lie bialgebras.
Similarly,  $i_2 : A^* \rightarrow A \oplus A^*$
is also a homomorphism of mock-Lie  bialgebras
since $\Delta_{D(A)}(f_i)=\Delta_{A^*}(f_i)$.
\end{proof}
 \begin{rmk}
 With the above mock-Lie bialgebra structure given in Theorem \ref{double}, $A\oplus A^*$
is called the  double of $A$. We denote it
by $D(A)$.
 \end{rmk}
%%%%%%%%%%%%%%%%%%%%%%%%%%%%%
\section{$\mathcal{O}$-operators of mock-Lie algebras and mock-Lie Yang-Baxter equation}
%%%%%%%%%%%%%%%%%%%%%%%%%%%%%
In this section, we interpret solution of mock-Lie Yang-Baxter equation in term of  $\mathcal{O}$-operators $($see \cite{Kupershmidt}$)$.
Let $V$ be a vector space. For any $r\in V\otimes V$, $r$ can be regarded as a map
from $V^*$ to $V$ in the following way:
\begin{equation}\label{formlinr}
    \langle u^*,r( v^*)\rangle=\langle u^*\otimes  v^*,r\rangle,\quad \forall u^*,v^*\in V^*,
\end{equation}
where $\langle\cdot,\cdot\rangle$ is the ordinary pairing between the vector space $V$ and the dual
space $V^*$. The tensor $r\in V\otimes V$ is called nondegenerate if the above induced
linear map is invertible. Moreover, any invertible linear map $T : V^*\rightarrow V$
induces a nondegenerate bilinear form $\omega(, )$ on $V$ by
\begin{equation}
    \omega(u,v)=\langle T^{-1}(u),v\rangle,\quad \forall u,v\in V.
\end{equation}
\begin{defi}\cite{Baklouti}
A \textbf{symplectic form} on a mock-Lie algebra $(A, \c)$ is a skew-symmetric
non-degenerate bilinear form $\omega$ satisfying
\begin{equation}
\omega(x\bullet y, z) + \omega(y\bullet z, x) + \omega(z\bullet x, y) = 0,\quad \forall x, y, z \in A.
\end{equation}
A mock-Lie algebra is called symplectic if it is endowed with a symplectic form.
\end{defi}
\begin{pro}\label{solution1}
Let $(A,\bullet)$ be a mock-Lie algebra and $r\in A\otimes A$ be skew-symmetric. Then $r$ is a solution of mock-Lie YBE in $A$ if and only if $r$ satisfies
\begin{equation}\label{solution}
    r(\xi)\bullet  r(\eta)=r\big(L^*(r(\xi))\eta+L^*(r(\eta))\xi\big),\quad \forall \xi,\eta \in A^*.
\end{equation}
\end{pro}
\begin{proof}
Let $\{e_1,\cdots,e_n\}$ be a basis of $A$ and $\{e^*_1,\cdots,e^*_n\}$ be the dual basis. Since $r$ is skew-symmetric, we can set $r=\sum_{1\leq i,j\leq n}a_{ij}e_i\otimes e_j,\;a_{ij}=-a_{ji}$. Suppose that $e_i\bullet e_j=\sum_{k=1}^{n}C^k_{ij}e_k$, where $C^k_{ij}$’s are the structure coefficients of mock-Lie algebra $A$ on the basis $\{e_1,\cdots,e_n\}$, then we get
\begin{align*}
    &r_{12}\bullet r_{13}=\big(\sum_{1\leq i,j\leq n}a_{ij}e_i\otimes e_j\otimes 1\big)\bullet \big(\sum_{1\leq p,q\leq n}a_{pq}e_p\otimes 1\otimes e_q\big)\\
    &\quad \quad \quad \;\;=\sum_{1\leq i,j,p,q,k\leq n}C^k_{ip}a_{ij}a_{pq}e_k\otimes e_j\otimes e_q;\\
    &r_{13}\bullet r_{23}=\sum_{1\leq i,j,p,q,k\leq n}C^k_{jq}a_{ij}a_{pq}e_i\otimes e_p\otimes e_k;\\
    &r_{12}\bullet r_{23}=\sum_{1\leq i,j,p,q,k\leq n}C^k_{jp}a_{ij}a_{pq}e_i\otimes e_k\otimes e_q.
\end{align*}
Then $r$ is a solution of mock-Lie YBE in $A$ if and only if
\begin{equation*}
   \sum_{1\leq i,p\leq n}\big(C^k_{ip}a_{ij}a_{pq}+C^q_{pi}a_{kp}a_{ji}-C^j_{ip}a_{ki}a_{pq}\big) e_k\otimes e_j\otimes e_q. 
\end{equation*}
On the other hand, by Eq. \eqref{formlinr}, we get $r(e^*_j)=\sum_{i=1}^{n}a_{ij}e_i=-\sum_{i=1}^{n}a_{ji}e_i,\;1\leq j\leq n$. If we take $\xi=e^*_j$, $\eta=e^*_q$ and by Eq. \eqref{solution}, we get \begin{equation*}
   \sum_{1\leq i,p\leq n}\big(C^k_{ip}a_{ij}a_{pq}+C^q_{pi}a_{kp}a_{ji}-C^j_{ip}a_{ki}a_{pq}\big) e_k=0. 
\end{equation*}
Therefore, it is easy to see that $r$ is a solution of mock-Lie YBE in $A$ if and only
if $r$ satisfies Eq. \eqref{solution}.
\end{proof}
\begin{ex}
Let $(A,\bullet)$ be a mock-Lie algebra. Then   a Rota-Baxter operator (of weight zero) is an $\mathcal{O}$-operator of $A$ associated to the adjoint representation $(A,L)$ and a skew-symmetric solution of mock-Lie YBE in $A$ is an $\mathcal{O}$-operator of $A$
associated to the representation $(A^*,L^*)$. 
\end{ex}
\begin{cor}
Let $(A,\bullet)$ be a mock-Lie algebra and $r\in A\otimes A$ be skew-symmetric. Suppose that there is a symmetric nondegenerate invariant bilinear
form $\omega$ on $A$. Let $\phi:A\rightarrow A^*$ a linear map given by $\langle \phi(x),y\rangle=\omega(x,y)$ for any $x,y\in A$. Then $r$ is a solution of mock-Lie YBE if and only if $r\phi$ is a Rota-Baxter operator (of weight zero) on $A$.
\end{cor}
\begin{proof}
 For any $x,y\in A$, we have $\phi(L(x)y)=L^*(x)\phi(y)$ since 
 \begin{align*}
     &\langle \phi(L(x)y),z\rangle =\omega(x \bullet y,z)=\omega(y \bullet x,z)\\
     &\quad \quad \quad \quad \quad \;\;=\omega(y,x\bullet z)=\langle L^*(x)\phi(y),z\rangle,\quad \forall x,y,z\in A.
 \end{align*}
 That is, the representations $(A,L)$ and $(A^*,L^*)$ are isomorphic. Let $\xi=\phi(x),\eta=\phi(y)$, then by Proposition \ref{solution1}, $ r$ is a solution of mock-Lie YBE in $A$ if and
only if
\begin{align*}
    &r\phi(x)\bullet r\phi(y)=r(\xi)\bullet r(\eta)=r\Big(L^*(r(\xi))\eta+L^*(r(\eta))\xi\Big)\\
    &\quad \quad\;\;\quad \quad \quad \;\quad \quad\;\;\quad \quad\;\quad=r\phi\Big(r\phi(x)\bullet y+x\bullet r\phi(y)\Big).
\end{align*}
Therefore the conclusion holds.
\end{proof}
Let $(A,\bullet)$ be a mock-Lie algebra. Let $(V,\rho)$ be a representation of $A$ and $\rho^*:A\rightarrow gl(V^*)$ be the dual representation. A linear map $T:V\rightarrow A$ can be identified as an element in $A\otimes V^*\subset (A\ltimes_{\rho^*}V^*)\otimes (A\ltimes_{\rho^*}V^*)$ as follows: Let $\{e_1,\cdots,e_n\}$ be a basis of $A$, let $\{v_1,\cdots,v_m\}$ be a basis of $V$ and $\{v^*_1,\cdots,v^*_m\}$ be its dual space of $V^*$. We set 
\begin{equation*}
    T(v_i)=\sum_{k=1}^{n}a_{ik}e_k,\;i=1,\cdots,m.
\end{equation*}
Since as vector space,  $Hom(V,A)\cong A\otimes V^*$, then we have 
\begin{equation}
    T=\sum_{i=1}^{m}T(v_i)\otimes v^*_i=\sum_{i=1}^{m}\sum_{k=1}^{n}a_{ik}e_k\otimes v^*_i\in A\otimes V^*\subset (A\ltimes_{\rho^*}V^*)\otimes (A\ltimes_{\rho^*}V^*). 
\end{equation}
\begin{thm}\label{solution2}
With the above notations, $r=T-\tau(T)$ is a skew-symmetric solution of the mock-Lie YBE in the semi-direct product mock-Lie algebra $(A\ltimes_{\rho^*}V^*)$ if and only if $T$ is an $\mathcal{O}$-operator associated to $(V,\rho)$.
\end{thm}
\begin{proof}
 We have 
 \begin{equation*}
    r=T-\tau(T)= \sum_{i=1}^{m}T(v_i)\otimes v^*_i-\sum_{i=1}^{m}v^*_i\otimes T(v_i),
 \end{equation*}
 thus we obtain
 \begin{align*}
    &r_{12}\bullet r_{13}=\sum_{1\leq i,j\leq m}\Big(Tv_i\bullet Tv_j\otimes v^*_i\otimes v^*_j-\rho^*(Tv_i)v^*_j\otimes v^*_i\otimes Tv_j-\rho^*(Tv_j)v^*_i\otimes Tv_i\otimes v^*_j\Big);\\
    &r_{12}\bullet r_{23}=\sum_{1\leq i,j\leq m}\Big(-v^*_i\otimes Tv_i\bullet Tv_j\otimes  v^*_j+Tv_i\otimes \rho^*(Tv_j)v^*_i\otimes v^*_j+v^*_i\otimes \rho^*(Tv_i)v^*_j\otimes Tv_j\Big);\\
    &r_{13}\bullet r_{23}=\sum_{1\leq i,j\leq m}\Big(v^*_i\otimes v^*_j\otimes Tv_i\bullet Tv_j-Tv_i\otimes v^*_j\otimes \rho^*(Tv_j)v^*_i-v^*_i\otimes Tv_j\otimes \rho^*(Tv_i)v^*_j\Big).
 \end{align*}
 By the definition of dual representation, we know
 \begin{equation*}
     \rho^*(Tv_j)v^*_i=\sum_{p=1}^m\langle v^*_i,\rho(Tv_j)v_p\rangle v^*_p.
 \end{equation*}
 Then
 \begin{align*}
    &\sum_{1\leq i,j\leq m}Tv_i\otimes \rho^*(Tv_j)v^*_i\otimes v^*_j =\sum_{1\leq i,j,p\leq m} \langle v^*_p,\rho(Tv_j)v_i\rangle Tv_p\otimes v^*_i\otimes v^*_j\\
    &= \sum_{1\leq i,j\leq m}T\Big(\langle v^*_p,\rho(Tv_j)v_i\rangle v_p\Big)\otimes v^*_i\otimes v^*_j=\sum_{1\leq i,j\leq m}T\Big(\rho(Tv_j)v_i\Big)\otimes v^*_i\otimes v^*_j.
 \end{align*}
 Then we get
 
   \begin{align*}
    &\;\;\;\;r_{12}\bullet r_{13}=\sum_{1\leq i,j\leq m}\Big(Tv_i\bullet Tv_j\otimes v^*_i\otimes v^*_j-v^*_i\otimes v^*_j\otimes T(\rho( Tv_j)v_i)-v^*_i\otimes T(\rho(Tv_j)v_i)\otimes v^*_j\Big);\\
    &-r_{12}\bullet r_{23}=\sum_{1\leq i,j\leq m}\Big(v^*_i\otimes Tv_i\bullet Tv_j\otimes  v^*_j-T(\rho(Tv_j)v_i)\otimes v^*_i\otimes v^*_j-v^*_i\otimes v^*_j\otimes T(\rho(Tv_i)v_j)\Big);\\
    &\;\;\;\;r_{13}\bullet r_{23}=\sum_{1\leq i,j\leq m}\Big(v^*_i\otimes v^*_j\otimes Tv_i\bullet Tv_j-T(\rho(Tv_i)v_j)\otimes v^*_i\otimes v^*_j-v^*_i\otimes T(\rho(Tv_i)v_j)\otimes v^*_j\Big).
 \end{align*}  
Hence, we get
\begin{align*}
  &r_{12}\bullet r_{13}+ r_{13}\bullet r_{23}- r_{12}\bullet r_{23}\\
  &=\sum_{1\leq i,j\leq m}\Bigg\{\Big(Tv_i\bullet Tv_j-T(\rho(Tv_i)v_j)-T(\rho(Tv_j)v_i)\Big)\otimes v^*_i\otimes v^*_j\\
  &\quad \quad \quad \quad \quad + v^*_i\otimes v^*_j\otimes\Big(Tv_i\bullet Tv_j-T(\rho(Tv_i)v_j)-T(\rho(Tv_j)v_i)\Big)\\
  &\quad \quad \quad \quad \quad +v^*_i\otimes \Big(Tv_i\bullet Tv_j-T(\rho(Tv_i)v_j)-T(\rho(Tv_j)v_i)\Big)\otimes v^*_j\Bigg\}.
\end{align*}
So $r$ is a  solution of the mock-Lie YBE in the semi-direct product mock-Lie algebra $(A\ltimes_{\rho^*}V^*)$ if and only if $T$ is an $\mathcal{O}$-operator associated to $(V,\rho)$.
\end{proof}
Combining Proposition \ref{solution1} and Theorem \ref{solution2}, we have the following conclusion.
\begin{cor}
Let $(A,\bullet)$ be a mock-Lie algebra and $(V,\rho)$ be a representation of $A$. Set $\widehat{A}=A\ltimes_{\rho^*}V^*$. Let $T:V\rightarrow A$ be a linear map. Then the
following conditions are equivalent:
\begin{enumerate}
    \item $T$ is an $\mathcal{O}$-operator of $A$ associated to $(V,\rho)$.
    \item $T-\tau(T)$ is a skew-symmetric solution of the mock-Lie YBE in the Jordan algebra $\widehat{A}$.
    \item $T-\tau(T)$ is an $\mathcal{O}$-operator of the mock-Lie algebra $\widehat{A}$ associated to $(\widehat{A}^*,L_{\widehat{A}}^*)$.
\end{enumerate}
\end{cor}
\begin{rmk}
The equivalence between the above $(1)$ and $(3)$ can be obtained
by a straightforward proof and then Theorem \ref{solution2} follows from this
equivalence and Proposition \ref{solution1}.
\end{rmk}
The following conclusion reveals the relationship between  mock-pre-Lie algebras
and the mock-Lie algebras with a  symplectic form:
\begin{pro}
Let $(A,\bullet)$ be a mock-Lie algebra with a  symplectic form $\omega$. Then there exists a compatible  pre-mock-Lie algebra structure $"\c"$ on $A$ given by 
\begin{equation}
    \omega(x\c y,z)=\omega(y,x\bullet z),\quad \forall x,y,z\in A.
\end{equation}
\end{pro}
\begin{proof}
 Define a linear map $T:A\rightarrow A^*$ by $\langle T(x),y\rangle=\omega(x,y)$ for any $x,y\in A$. For any $\xi,\eta,\gamma\in A^*$, since $T$ is invertible, there exist $x,y,z\in A$ such that $Tx=\xi,Ty=\eta,Tz=\gamma$. Then $T^{-1}:A^*\rightarrow A$ is an $\mathcal{O}$-operator of $A$ associated to $(A^*,L^*)$ since for any $x,y,z\in A$, we have
 \begin{align*}
     &\langle T(x\bullet y),z\rangle=\omega(x\bullet y,z)=\omega(y,x\bullet z)+\omega(x,y\bullet z)\\
     &\quad \quad \quad\quad  \quad \;=\langle L^*(x)T(y),z\rangle +\langle L^*(y)T(x),z\rangle. 
 \end{align*}
 By Proposition \ref{compremock}, there is a compatible  mock-pre-Lie algebra structure $"\c"$ on
$A$ given by
\begin{equation*}
    x\c y=T^{-1}\big(L^*(x)T(y)\big),\quad \forall x,y\in A,
\end{equation*}
which implies that 
\begin{align*}
    &\omega(x\c y,z)=\langle T(x\c y),z\rangle =\langle L^*(x)T(y),z\rangle\\
    &\quad \quad \quad\quad\;=\langle T(y),x\bullet z\rangle =\omega(y,x\bullet z),\quad \forall x,y,z\in A.
\end{align*}
Hence the proof.
\end{proof}
The following conclusion provides a construction of solutions of mock-Lie YBE in
certain mock-Lie algebras from  mock-pre-Lie algebras.
\begin{cor}
Let $(A,\c)$ be a  mock-pre-Lie algebra. Let $\{e_1,\cdots,e_n\}$ be a basis of $A$ and $\{e^*_1,\cdots,e^*_n\}$ the dual basis. Then 
\begin{equation}
    r=\sum_{i=1}^{n}(e_i\otimes e^*_i-e^*_i \otimes e_i)
\end{equation}
is a skew-symmetric solution of the mock-Lie YBE in the mock-Lie algebra $(A^{ac})\ltimes_{\Theta^*}(A^{ac})^*$.
\end{cor}
\begin{proof}
It follows from Theorem \ref{solution2} and the fact that the identity map $id$ is an $\mathcal{O}$-operator of
the sub-adjacent mock-Lie algebra $A^{ac}$ of a  mock-pre-Lie algebra associated to the representation $(A, \Theta)$.
\end{proof}
%%%%%%%%%%%%%%%%%%%%%%


\begin{thebibliography}{9999}
\bibitem{M.G}\textsc{M. Aguiar}, \textit{Infinitesimal Hopf algebras,}
Contemporary Mathematics 267, Amer. Math. Soc., (2000) 1-29.

\bibitem{Attan}
\textsc{S. Attan,}  \textit{Representations and O-operators of Hom-(pre)-Jacobi-Jordan algebras}. arXiv preprint  arXiv:2105.14650 (2021).

\bibitem{Agore1}
\textsc{A.L Agore, and G. Militaru,} \textit{On a type of commutative algebras,} Linear Algebra and its Applications, (2015) 485, 222-249.

\bibitem{Agore2}
\textsc{A.L Agore, and G. Militaru,} \textit{Algebraic constructions for Jacobi-Jordan algebras,} Linear Algebra and its Applications, (2021) 630, 158-178.

\bibitem{Bai1}
\textsc{C. Bai}, \textit{ Left-symmetric bialgebras and an analogue of the classical Yang–Baxter equation,} Communications in Contemporary Mathematics, 10(02), (2008), 221-260.

\bibitem{Bai2}
\textsc{C. Bai},  \textit{Double constructions of Frobenius algebras, Connes cocycles and their duality,} J. Noncommut. Geom.,
4, (2010), pp. 475-530.

\bibitem{B} \textsc{C. Bai}, \textit{ A unified algebraic approach to the classical Yang-Baxter equation,}J. Phys. A: Math. Theor.
40 (2007), 11073-11082.

\bibitem{B.G.T} \textsc{ C.  Bai, L. Guo and T. Ma,}  \textit{Bialgebras, Frobenius algebras and associative Yang-Baxter equations for
Rota-Baxter algebras},  arXiv:2112.10928 (2021).

\bibitem{B.L} \textsc{C. Bai, L. GUO, G. LIiu, and T. MA}, \textit{ Rota-Baxter Lie Bialgebras, Classical Yang-Baxter
Equations Aand Special L-Dendriform Bialgebras}, arXiv:2207.08703 

\bibitem{B.B.G} \textsc{ C. Bai, O. Bellier, L. Guo and X. Ni}, \textit{ Splitting of operations, Manin products and Rota-Baxter
operators.} International Mathematics Research Notices, 2013(3), 485-524.

\bibitem{Baklouti}
\textsc{A. Baklouti, S. Benayadi,} \textit{ Symplectic Jacobi-Jordan algebras,} Linear and Multilinear Algebra. (2021), 69(8), 1557-1578.

\bibitem{Benayadi}
\textsc{A. Baklouti, S. Benayadi,A. Makhlouf,  S. Mansour,} \textit{ Cohomology and deformations of Jacobi-Jordan algebras.} ArXiv: 2109.12364 (2021).

\bibitem{BasdouriFadousMabroukMakh}
 \textsc{I. Basdouri,M. Fadous, S. Mabrouk, A. Makhlouf}, \textit{Poisson superbialgebras}. arXiv preprint arXiv:2205.06184 (2022).
 
\bibitem{R.J1}\textsc{R.J.   Baxter}, \textit{Partition function fo the eight-vertex lattice model }. Ann. Physics 70 (1972) 193-228.

\bibitem{R.J2}\textsc{R.J.   Baxter}, \textit{Exactly solved models in statistical mechanics}. Academic Press, London, 1982.
 
\bibitem{R.J.B} \textsc{R. J. Baxter}, \textit{ Partition function of the eight-vertex lattice model,}Ann. Phys. 70 (1972), 193-228.

\bibitem{D.B}\textsc{D. Burde}, \textit{ Left-symmetric algebras and pre-Lie algebras in geometry and physics}. Cent. Eur. J. Math. 4
(2006), 323-357.

\bibitem{B.F} \textsc{D. Burde, A. Fialowski}, \textit{ Jacobi-Jordan algebras,} Lin. Algebra Appl.495(2014),586-594.

\bibitem{V.A}
\textsc{V. Chari, A. Pressley,} \textit{ A guide to quantum groups,} Cambridge University Press, Cambridge (1994).

\bibitem{Dassoundo2}
 \textsc{M. L. Dassoundo,} \textit{Pre-anti-flexible bialgebras}. Communications in Algebra, 49(9) (2021), 4050-4078.

\bibitem{Dassoundo1}
\textsc{M. L. Dassoundo, C. Bai and  M. N.Hounkonnou,} \textit{Anti-flexible bialgebras.} Journal of Algebra and Its Applications,(2021) 2250212.

\bibitem{Drinfeld}
\textsc{V.G. Drinfeld,} \textit{ Hamiltonian structures of Lie groups, Lie bialgebras and the geometric meaning of the classical
Yang-Baxter equations}. Soviet Math. Dokl., 27, (1983), pp. 68-71.

\bibitem{FadousSamiMakhlouf}
\textsc{M. Fadous, , S. Mabrouk, A. Makhlouf}, \textit{  On Hom-Lie superbialgebras}. Communications in Algebra, 47(1) (2019), 114-137. 

\bibitem{Haliya}
\textsc{C. E. Haliya, G. D. Houndedji}, \textit{ On a pre-Jacobi-Jordan algebra: relevant properties and double construction}. arXiv preprint arXiv:2007.06736 (2020).

\bibitem{C.G}\textsc{C. E. Haliya, G. D. Houndedji}, \textit{Hom-Jacobi-Jordan and Hom-antiassociative
algebras with symmetric invariant nondegenerate
bilinear forms
},Quasigroups and Related Systems 29 (2021), 61 -88.

\bibitem{Hou}
\textsc{D. Hou, X. Ni, C. Bai}, \textit{Pre-Jordan algebras}. Mathematica Scandinavica, (2013) 19-48.

\bibitem{S.G}
\textsc{ S.A. Joni, G.C. Rota,} \textit{Coalgebras and bialgebras in combinatories,} Stud. Appl. Math. 61 (1979), 93-139.

\bibitem{Kupershmidt}
\textsc{B.A. Kupershmidt,} \textit{ What a classical r-matrix really is,} J. Nonlinear. Math. Phys., 6, (1999), pp. 448-488.

\bibitem{Makhlouf1}
\textsc{A. Makhlouf,F.  Panaite},  \textit{Yetter-Drinfeld modules for Hom-bialgebras}. Journal of Mathematical Physics, 55(1) (2014), 013501.

\bibitem{MaMakhloufSilvestrov} \textsc{ T. Ma, A. Makhlouf,  S. Silvestrov},\textit{ Rota–Baxter cosystems and coquasitriangular mixed bialgebras}. Journal of Algebra and Its Applications, 20(04) (2021), 2150064.

\bibitem{J.H} \textsc{J.H.H.  Perk, H. Au-Yang, } \textit{Yang-Baxter equations}, , arXiv:math-ph/0606053v1.

\bibitem{G.C.R} \textsc{G. C. Rota}, \textit{ Baxter operators, an introduction,} In: Gian-Carlo Rota on Combinatorics: Introductory
Papers and Commentaries, Joseph P. S. Kung, Editor, Birkh¨auser, Boston, 1995.

\bibitem{C.N} \textsc{ C.N.  Yang}, \textit{ Some exact results for the many-body problem in one dimension with replusive delta-function
interaction, } Phys. Rev. Lett. 19 (1967) 1312-1315.

\bibitem{Zhelyabin}
\textsc{V. N. Zhelyabin,} \textit{Jordan bialgebras and their relation to Lie bialgebras}. Algebra Logic, 36, (1997), pp. 1-15.

\bibitem{Zhelyabin2}
\textsc{V. N. Zhelyabin}, \textit{Jordan D-bialgebras and symplectic forms on Jordan algebras}, Siberian
Adv. Math 10:2 (2000), 142–150.

\bibitem{Zhelyabin3}
\textsc{V. N. Zhelyabin},\textit{ On a class of Jordan D-bialgebras}, St. Petersburg Math. J. 11 (2000),
589–609.


\bibitem{P.Z} \textsc{P.  Zumanovich}, \textit{Special and exceptional mock-Lie algebras,} Linear Algebra Appl. 518 (2017),79-96.







 





















\end{thebibliography}
\end{document}